%% file: EKRMatchingsExtOnly0312.tex
\documentclass[10pt]{article}

\usepackage[latin1]{inputenc}
\usepackage{amsfonts}
\usepackage{amsmath}
\usepackage{amssymb}
\usepackage{amsthm}
\usepackage{mathrsfs}
\usepackage[american]{babel}
\usepackage{graphicx}
\usepackage{fullpage}
\newtheorem{theorem}{Theorem}[section]

\newtheorem{lemma}[theorem]{Lemma}
\newtheorem{claim}[theorem]{Claim}

\newtheorem{cor}[theorem]{Corollary}

\author{\ \\ \\
Vikram Kamat \thanks{\texttt{vkamat@csa.iisc.ernet.in}}
\ and
Neeldhara Misra \thanks{\texttt{mail@neeldhara.com}}\\
{\small Department of Computer Science \& Automation,}\\
{\small Indian Institute of Science, Bangalore -- 560 012, India.}\\ \\ \\
}

\title{An Erd\H{o}s--Ko--Rado theorem for matchings in the complete graph}
\begin{document}

\maketitle

\begin{abstract}
We consider the following higher-order analog of the Erd\H{o}s--Ko--Rado theorem. For positive integers $r$ and $n$ with $r\leq n$, let $\mathcal{M}^r_n$ be the family of all matchings of size $r$ in the complete graph $K_{2n}$. For any edge $e\in E(K_{2n})$, the family $\mathcal{M}^r_n(e)$, which consists of all sets in $\mathcal{M}^r_n$ containing $e$ is called the star centered at $e$. We prove that if $r<n$ and $\mathcal{A}\subseteq \mathcal{M}^r_n$ is an intersecting family of matchings, then $|\mathcal{A}|\leq |\mathcal{M}^r_n(e)|$, where $e\in E(K_{2n})$. We also prove that equality holds if and only if $\mathcal{A}$ is a star. The main technique we use to prove the theorem is an analog of Katona's elegant cycle method.


\end{abstract}

\section{Introduction}

\input{introduction.tex}

\section{Baranyai Partitions}\label{baranyai}

\input{partitions.tex}

\section{Proof of Main Theorem}\label{main}

\input{maintheorem.tex}

\end{document}

%% file: introduction.tex
For a positive integer $n$, let $[n]=\{1,2,\ldots,n\}$ and let ${[n] \choose r}$ be the family of all $r$-subsets of $[n]$. We say that a family $\mathcal{F}\subseteq {[n] \choose r}$ is $t$-intersecting if for any two sets in $A, B\in \mathcal{F}$, $|A\cap B|\geq t$. We say that a family is \textit{intersecting} if it is $1$-intersecting. One of the seminal results in extremal set theory, due to Erd\H{o}s, Ko and Rado gives a best-possible upper bound of the size of \textit{uniform} intersecting families and also characterizes the structures which attain this bound.
\begin{theorem}[Erd\H{o}s--Ko--Rado]\label{ekr}
Let $\mathcal{F}\subseteq {[n] \choose r}$ be intersecting. If $r\leq n/2$, then $|\mathcal{F}|\leq {n-1 \choose r-1}$. If $r<n/2$, then equality holds if and only if $\mathcal{F}=\{A\in {[n] \choose r}:x\in A\}$, for $x\in [n]$.
\end{theorem}

In this paper, we are concerned with a higher-order analog of Theorem \ref{ekr}. Informally, a higher-order extremal problem is one in which the elements of a family are families of sets (as opposed to sets). A good overview of higher-order extremal problems is provided by Ahlswede, Cai and Zhang \cite{acz}. P.L. Erd\H{o}s and Sz\'ekely \cite{esz} also consider higher-order Erd\H{o}s--Ko--Rado theorems. One of the problems considered in that paper is the case in which every set system is a set partition and an Erd\H{o}s--Ko--Rado type result for $t$-intersecting systems of set partitions of $[n]$ is proved for sufficiently large $n$. Meagher and Moura \cite{mm} consider the problem for $t$-intersecting systems of \textit{uniform} $k$-partitions of $[n]$, i.e. set partitions with $k$ blocks where each block has $n/k$ elements, and prove an Erd\H{o}s--Ko--Rado theorem when $n$ is sufficiently large. In addition, they also establish a result for intersecting uniform set partitions for all $n$.

In this paper, we consider intersecting families of \textit{partial} uniform set partitions. In particular, we restrict our attention to the case when each block of the partition has two elements. It is convenient to think of this problem in terms of intersecting families of matchings in the complete graph on an even number of vertices. Before we state our main result, we introduce some notation.

For $n\geq 1$, let $K_{2n}$ be the complete graph on $2n$ vertices with vertex set $V=V(K_{2n})=[2n]$ and edge set $E=E(K_{2n})={[2n] \choose 2}$. For $r\leq n$, let $\mathcal{M}^r_n$ be the family of all matchings in $K_{2n}$ with $r$ edges (henceforth, we call them $r$-matchings). Analogous to the notion of intersection for set systems, we say that $\mathcal{A}\subseteq \mathcal{M}^r_n$ is intersecting if any two $r$-matchings in $\mathcal{A}$ share at least one edge. For an edge $e\in E$, let $\mathcal{M}^r_n(e)=\{A\in \mathcal{M}^r_n:e\in A\}$, called a star centered at edge $e$. Let $\chi(n,r)=\dfrac{{2n \choose 2}\ldots{2n-2(r-1) \choose 2}}{r!}$ and $\phi(n,r)=\dfrac{{2n-2 \choose 2}\ldots{2n-2(r-1) \choose 2}}{(r-1)!}$. Note that $\phi(n,r)=\dfrac{r}{{2n \choose 2}}\chi(n,r)$. It is not hard to see that $|\mathcal{M}^r_n|=\chi(n,r)$ and for any $e\in E$, $|\mathcal{M}^r_n(e)|=\phi(n,r)$. We can now state our main result.
\begin{theorem}\label{mainthm}
For $r<n$, if $\mathcal{A}\subseteq \mathcal{M}^r_n$ is an intersecting family of $r$-matchings, then $|\mathcal{A}|\leq \phi(n,r)$ with equality holding if and only if $\mathcal{A}=\mathcal{M}^r_n(e)$ for $e\in E$.
\end{theorem}

We note that the case $r=n$ is settled (as part of a stronger theorem for uniform set partitions) by Meagher and Moura \cite{mm}. To prove Theorem \ref{mainthm}, we use an analog of Katona's cycle method \cite{katona}. As in Katona's original proof of the Erd\H{o}s--Ko--Rado theorem, the main challenge is to come up with a class of objects over which to carry out the double counting argument. We will use the notion of Baranyai partitions to construct these objects.

The rest of the paper is organized as follows. In Section \ref{baranyai}, we describe the class of objects used in the Katona-style argument and prove a crucial property about this class that enables us to obtain the required bound in Theorem \ref{mainthm}. In Section \ref{main}, we give a proof of Theorem \ref{mainthm} that includes the tight upper bound and characterization of the extremal structures.

%% file: partitions.tex

A \textit{Baranyai partition} of the complete graph $K_{2n}$ is a partitioning of its edge-set $E$ into ($2n-1$) perfect matchings. We first describe a well-known construction for a Baranyai partition of $K_{2n}$. For any $i\in [2n-1]$ and $1\leq j\leq n-1$, let $e^j(i)=\{i+j,i-j+2n-1\}$\footnote{Addition is carried out modulo $2n-1$, so $a+b$ is either $a+b$ or $a+b-(2n-1)$, depending on which lies in $[2n-1]$.}, i.e. the edge between the vertices $i+j$ and $i-j+2n-1$. Also, let $e^0(i)=\{i,2n\}$. Finally, for each $1\leq i\leq 2n-1$, let $M(i)=\{e^1(i),e^2(i),\ldots,e^{n-1}(i),e^0(i)\}$. Now $\mathcal{B}=\{M(1),\ldots,M(2n-1)\}$, is a Baranyai partition of $K_{2n}$.


We define a \textit{rooted} Baranyai partition as one with a specific vertex $i\in [2n]$ identified as the root of the partition. (We call it an $i$-rooted Baranyai partition.) We also define a \textit{rooted Baranyai order} (henceforth referred to as a \textit{rooted order}) as an ordering of the parts of a rooted Baranyai partition. (Note that the individual parts themselves are still considered as unordered sets, though we will later work with a fixed ordering of the individual parts.) Now, using the construction described above, we construct a rooted order corresponding to each permutation $\sigma\in S_{2n}$. Let $\sigma=(\sigma(1),\ldots,\sigma(2n))\in S_{2n}$ be a permutation of $[2n]$. For any $i\in [2n-1]$ and $1\leq j\leq n-1$, let $e_{\sigma}^j(i)=\{\sigma(i+j),\sigma(i-j+2n-1)\}$, the edge between the two vertices $\sigma(i+j)$ and $\sigma(i+2n-1-j)$. Also, let $e_{\sigma}^0(i)=\{\sigma(i),\sigma(2n)\}$. Let $\mathcal{B}_{\sigma}=(M_{\sigma}(1),\ldots,M_{\sigma}(2n-1))$ be the rooted order corresponding to $\sigma$, with root $\sigma(2n)$ and $M_{\sigma}(i)=\{e_{\sigma}^1(i),e_{\sigma}^2(i),\ldots,e_{\sigma}^{n-1}(i),e_{\sigma}^0(i)\}$ for each $i\in [2n-1]$.

One possible way to visualize these constructions is the following. Let the points $\sigma(1),\ldots,\sigma(2n-1)$ be vertices of a regular $2n-1$-sided polygon and let $\sigma(2n)$ be the center of the common circumscribed circle on which these points lie. Consider the edge set of $K_{2n}$ to be the set of all line segments between pairs of the $2n$ points. Now for each $1\leq i\leq 2n-1$, the perfect matching $M_{\sigma}(i)$ consists of the line segment joining $\sigma(i)$ and $\sigma(2n)$ along with the $n-1$ line segments that are perpendicular to it. Figure~\ref{fig:baranyai} provides an example for $i=(2n-1)$.

For $\sigma,\mu\in S_{2n}$, we say that $\mathcal{B}_{\sigma}=\mathcal{B}_{\mu}$ if and only if the two partitions have the same root vertex and $M_{\sigma}(i)=M_{\mu}(i)$ for each $i\in [2n-1]$; in other words, $\mathcal{B}_{\sigma}=\mathcal{B}_{\mu}$ if and only if $\sigma=\mu$. Under this definition, it is clear that each permutation corresponds to a unique rooted order. See Figure~\ref{fig:partitions} for an example of the rooted order of $K_{8}$ corresponding to the identity permutation.

\begin{figure}[ht]
\begin{center}
\includegraphics{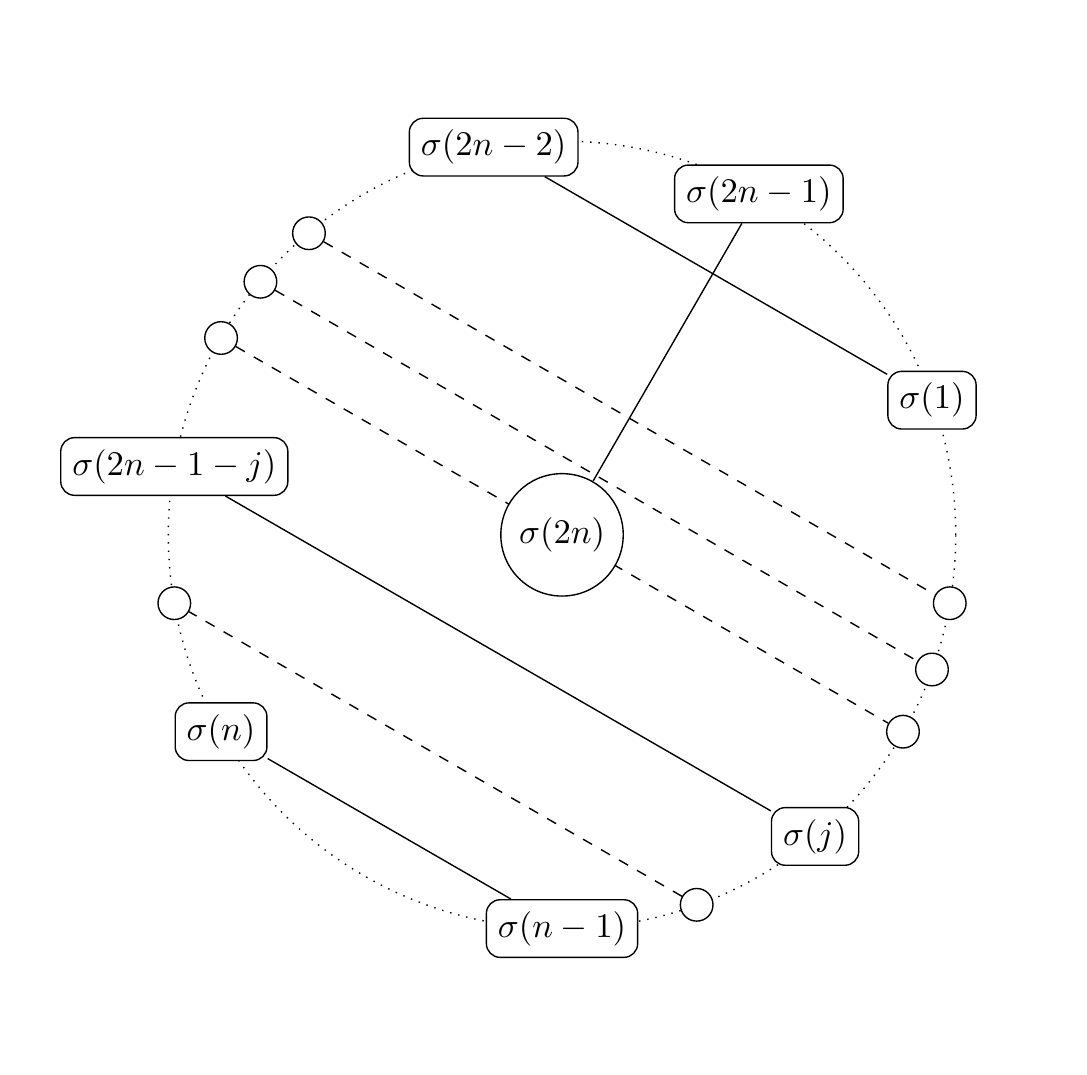}
\end{center}
\caption{A visual representation of the part $M_{\sigma}(2n-1)$ in $\mathcal{B}_{\sigma}$.}
\label{fig:baranyai}
\end{figure}

For each $1\leq i\leq 2n-1$, we fix an ordering of $M_{\sigma}(i)$, denoted by $\tilde{M}_{\sigma}(i)=(e_{\sigma}^{n-1}(i),e_{\sigma}^{n-2}(i),\ldots,e_{\sigma}^{1}(i),e_{\sigma}^0(i))$. Henceforth, we consider all parts of $\mathcal{B}_{\sigma}$ to be ordered in this manner; in other words, $\mathcal{B}_{\sigma}=(\tilde{M}_{\sigma}(1),\ldots,\tilde{M}_{\sigma}(2n-1))$. We can observe that $\mathcal{B}_{\sigma}$ corresponds to a \textit{cyclic order}, say $\psi_{\sigma}$, on ${[{2n \choose 2}]}$, where the edges in $\tilde{M}_{\sigma}(1)$ occupy the first $n$ positions of the cyclic order, the edges in $\tilde{M}_{\sigma}(2)$ occupy the next $n$ positions, and so on. More formally, let $\psi_{\sigma}(k)$ denote the element in the $k^{th}$ position of the cyclic order, for each $1\leq k\leq n(2n-1)$. Then for each $1\leq i\leq 2n-1$ and $1\leq j\leq n$, let $\psi_{\sigma}((i-1)n+j)=e_{\sigma}^{n-j}(i)$ (see Figure~\ref{fig:partitions}).

\begin{figure}[ht]
\begin{center}
\includegraphics{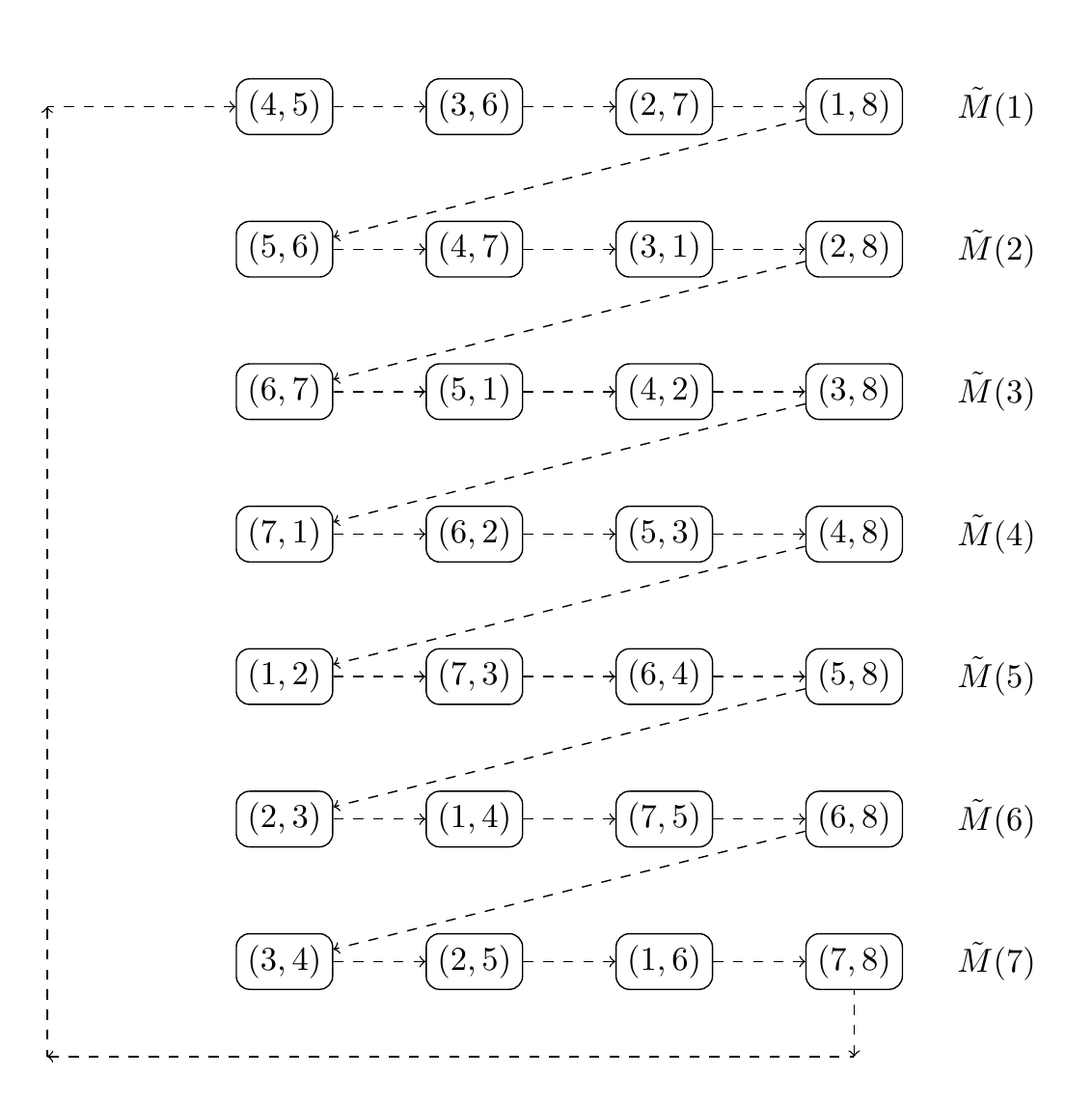}
\end{center}
\caption{A rooted Baranyai order for $n=4$, corresponding to the identity permutation. The dotted arrows indicate the cyclic order $\psi$.}
\label{fig:partitions}
\end{figure}


Now for a permutation $\pi\in S_{2n}$ and $c\in [2n-1]$, let $\pi_c$ be the permutation defined as follows. For each $1\leq i\leq 2n-1$, $\pi_c(i)=\pi(i+c)$ and $\pi_c(2n)=\pi(2n)$.
\begin{claim}\label{cyclic}
For each $1\leq i\leq 2n-1$, $\tilde{M}_{\pi_c}(i)=\tilde{M}_{\pi}(i+c)$.
\end{claim}
\begin{proof}
We wish to show that for each $0\leq j\leq n-1$, $e^j_{\pi_c}(i)=e^j_{\pi}(i+c)$. If $j=0$, then $e^0_{\pi_c}(i)=\{\pi_c(i),\pi_c(2n)\}=\{\pi(i+c),\pi(2n)\}=e^0_{\pi}(i+c)$, so let $1\leq j\leq n-1$. It can be seen that the following holds.
\begin{eqnarray*}
e^j_{\pi_c}(i) &=& \{\pi_c(i+j),\pi_c(i-j+2n-1)\} \\
&= & \{\pi(i+c+j),\pi(i+c-j+2n-1)\} \\
&= & e^j_{\pi}(i+c).
\end{eqnarray*}
\end{proof}


Next, we show that these cyclic orders satisfy a crucial property, namely that every sequence of $n-1$ consecutive edges in them correspond to matchings in $K_{2n}$. (Figure~\ref{fig:partitions} provides one example.) Formally, if $\psi$ is a cyclic order on $[{2n \choose 2}]$ and $i\in [n(2n-1)]$, let
$I^r_i(\psi)=\{\psi(i),\ldots,\psi(i+r-1)\}$ be the $r$-interval in $\psi$ beginning at index $i$. (In this case, addition is carried out modulo $n(2n-1)$.) Also, for an interval $I=(e_1,\ldots,e_r)$ lying in some $\psi_{\sigma}$, call $e_j$ the edge in position $j$ of the interval.
\begin{claim}\label{goodness}
If $\psi=\psi_{\sigma}$ is the cyclic order on $[{2n \choose 2}]$ corresponding to $\sigma\in S_{2n}$ and if $r\leq n-1$, then for each $i\in [n(2n-1)]$, $I^r_i(\psi)$ is a matching.
\end{claim}
\begin{proof}
We will assume that $r=n-1$, as this will imply that the result holds for each $1\leq r\leq n-1$. Without loss of generality, it will be sufficient to prove the claim when $\sigma$ is the identity permutation. Let the corresponding rooted order be $\mathcal{B}_{\sigma}=\mathcal{B}=(\tilde{M}(1),\ldots,\tilde{M}(2n-1))$ where $\tilde{M}(i)=\{e^{n-1}(i),e^{n-2}(i),\ldots,e^1(i),e^0(i)\}$. (As we are dealing with a fixed permutation, we drop the subscript $\sigma$ from the notation.) It is trivial to note that any interval contained entirely within a single $\tilde{M}(i)$, for some $1\leq i\leq 2n-1$, is a matching, so consider an interval of length $r$ that has a non-empty intersection with $\tilde{M}(i)$ and $\tilde{M}(i+1)$ for some $1\leq i\leq 2n-1$. Without loss of generality, suppose $i=2n-1$, so $i+1=1$, addition being done modulo $2n-1$. For each $0\leq j\leq n-3$, define the interval $I_j$ as follows.
$$I_j=\{\overbrace{\underbrace{e^j(2n-1),\ldots,e^1(2n-1)}_{j\text{ elements}},e^0(2n-1)}^{\in \tilde{M}(2n-1)},\underbrace{\overbrace{e^{n-1}(1),\ldots,e^{j+2}(1)}^{\in \tilde{M}(1)}}_{(n-j-2)\text{ elements.}}\}.$$ (Notice that if $j \in \{n-2,n-1\}$, then $I$ lies entirely in $\tilde{M}(2n-1)$.) We will show that $I_j$ is a matching for each $0\leq j\leq n-3$. This is clear when $j=0$, as $e^1(1)=\{2,2n-1\}$, so let $1\leq j\leq n-3$. Let $k,l$ be such that $1\leq k\leq j$ and $j+2\leq l\leq n-1$. Now $e^k(2n-1)=\{k,2n-(k+1)\}$ and $e^l(1)=\{l+1,2n-l\}$. As $l\geq k+2$, this implies $e^k(2n-1)\cap e^l(1)=\emptyset$, thus proving that the interval $I_j$ is a matching.
\end{proof}


\textbf{Remark:} Claim \ref{goodness} can also be described in terms of powers of Hamiltonian cycles in a particular Kneser graph. The Kneser graph $K(n,r)$ is defined as follows. The vertex set of $K(n,r)$ is ${[n] \choose r}$, and for any $A,B\in {[n] \choose r}$, there is an edge between $A$ and $B$ if and only if $A\cap B=\emptyset$. Now, for any graph $G$ on $m$ vertices and $k\geq 1$, we say that $G$ contains the $k^{th}$ power of a Hamiltonian cycle if there is a cyclic ordering of its vertices, say $(v_1,\ldots,v_m)$, such that for any $i,j\in [m]$ with $i<j$, if $\textrm{ min }\{j-i,i+m-j\}\leq k$, then there is an edge between vertices $v_i$ and $v_j$. By Claim \ref{goodness}, we have the following.

\begin{cor}\label{kneserham}
The Kneser graph $K(2n,2)$ contains a $k^{th}$ power of a Hamiltonian cycle for each $1\leq k\leq n-2$.
\end{cor}

%% file: maintheorem.tex

The statement of the theorem is trivial when $n\leq 2$, so let $n\geq 3$. For $r\leq n-1$, let $\mathcal{A}\subseteq \mathcal{M}^r_n$ be an intersecting family of $r$-matchings in $K_{2n}$. For $\sigma\in S_{2n}$, let $\mathcal{B}_{\sigma}$ be the unique rooted order and $\psi_{\sigma}$ be the cyclic order on $[{2n \choose 2}]$ corresponding to $\sigma$. We say that a matching $A\in \mathcal{M}^r_n$ is \textit{compatible} with $\sigma$ if some $r$-interval in $\psi_{\sigma}$ is $A$. Let $\mathcal{A}_{\sigma}=\{A\in \mathcal{A}:A \textrm{ is compatible with } \sigma\}$. Using a well-known argument of Katona, it can be shown that for any $\sigma\in S_{2n}$, $|\mathcal{A}_{\sigma}|\leq r$, with equality holding only when $\mathcal{A}_{\sigma}$ is a \textit{star}, i.e. for some $i\in [n(2n-1)]$, the edge $e=\psi_{\sigma}(i)$ is contained in every set in $\mathcal{A}_{\sigma}$. We say that $\sigma$ is \textit{saturated} and \textit{centered} at $e$ if $|\mathcal{A}_{\sigma}|= r$ and $e$ is the common edge that lies in every set in $\mathcal{A}_{\sigma}$. On the other hand, for a set $A\in \mathcal{M}^r_n$, let $Q_A=\{\sigma\in S_{2n}: A \textrm{ is compatible with } \sigma\}$ and let $q_A=|Q_A|$. We prove the following claim.
\begin{claim}\label{findq}
For any $A\in \mathcal{M}^r_n$, $q_A=n(2n-1)r!2^r(2n-2r)!$.
\end{claim}
\begin{proof}
Let $A=\{\{x_1,y_1\},\ldots,\{x_r,y_r\}\}$, and let $g(A)=\{x_1,\ldots,x_r\}\cup\{y_1,\ldots,y_r\}$. Let $Q^1_A=\{\sigma\in Q_A:\sigma(2n)\not\in g(A)\}$ and let $Q^2_A=Q_A\setminus Q^1_A$. Finally, let $q^1_A=|Q^1_A|$ and $q^2_A=|Q^2_A|$. We compute $q^1_A$ and $q^2_A$ separately. In both cases, our strategy will be to compute these numbers based on the position of $A$ in the corresponding rooted orders.
\\
\\
If $\sigma\in Q^1_A$, then $A$ will lie entirely in one part of $\mathcal{B}_{\sigma}$, since  the last edge in every part contains $\sigma(2n)$. Within this (ordered) part, $A$ will be an interval beginning at the edge in the $j^{th}$ position of the part, for some $1\leq j\leq n-r$. We count all permutations $\sigma$ where $A$ is an interval in $\mathcal{B}_{\sigma}$ lying entirely in the $i^{th}$ part of $\mathcal{B}_{\sigma}$, and beginning at position $j$ in the part, for some $i\in [2n-1]$ and $j\in [n-r]$. More formally, we count all permutations $\sigma$ such that $\{e_{\sigma}^{n-j}(i),e_{\sigma}^{n-j-1}(i),\ldots,e_{\sigma}^{n-j-r+1}(i)\}=\{\{x_1,y_1\},\ldots,\{x_r,y_r\}\}$. There are $r$ pairs of positions such that the permutation $\sigma$ maps each pair to a unique edge $\{x_k,y_k\}$, $1\leq k\leq r$. Clearly there are $r!2^r(2n-2r)!$ such permutations, and summing over all $(i,j)$ pairs gives us $q^1_A=(n-r)(2n-1)r!2^r(2n-2r)!$.
\\
\\
Next, we compute $q^2_A$. If $\sigma\in Q^2_A$, then there is an $i\in [2n-1]$ such that $e^0_{\sigma}(i)\in A$. For a given $i\in [2n-1]$, we first count all permutations $\sigma$ such that $e^0_{\sigma}(i)\in A$ and $A$ contains at most $r-1$ edges from the $(i+1)^{th}$ part of $\mathcal{B}_{\sigma}$. Suppose $A$ contains $j$ edges from $\tilde{M}_{\sigma}(i)$, for $1\leq j\leq r$, so

$$A=\{e_{\sigma}^{j-1}(i),\ldots,e_\sigma^1(i),e_{\sigma}^0(i),e_{\sigma}^{n-1}(i+1),\ldots,e_{\sigma}^{n-(r-j)}(i+1)\}.$$

The number of permutations in this case is also $r!2^r(2n-2r)!$, by the same argument as before. Summing over all values of $i$ and $j$ gives us $q^2_A=(2n-1)(r)(r!2^r(2n-2r)!)$.
\\
\\
Finally, $q_A=q^1_A+q^2_A=(n-r)(2n-1)r!2^r(2n-2r)!+r(2n-1)r!2^r(2n-2r)!=n(2n-1)r!2^r(2n-2r)!$, as required.
\end{proof}

Using Claim \ref{findq} and double counting, we can now give a proof of the bound in Theorem \ref{mainthm} as follows. We have $n(2n-1)r!2^r(2n-2r)!|\mathcal{A}|\leq r(2n)!$, which simplifies to $|\mathcal{A}|\leq \phi(n,r)$.
\\
\\
We now turn our attention to characterizing the extremal structures. Suppose that $|\mathcal{A}|=\phi(n,r)$. This implies that for every $\sigma\in S_{2n}$, $\sigma$ is saturated and hence centered at some edge in $E$. Throughout the rest of this section, we will assume that every $\sigma\in S_{2n}$ is saturated with respect to the family $\mathcal{A}$. We will show that every $\sigma$ is centered at the same edge, say $e\in E$, which will then imply that $\mathcal{A}=\mathcal{M}^r_n(e)$. Our basic strategy will be as follows. We first show that for every $c\in [2n-1]$, there exists a permutation $\sigma$ centered at the edge $e=\{\sigma(c),\sigma(2n)\}$. Next, we show that for $\sigma$ centered at $e=\{\sigma(c),\sigma(2n)\}$, any permutation $\mu$ that is obtained from $\sigma$ by an \textit{adjacent transposition}, i.e. by swapping a pair of elements in adjacent positions of $\sigma$, will also be centered at $e$. This will readily imply that every permutation will be centered at the same edge. We begin by making the following basic observation.

\begin{claim}\label{obs}
For $\pi\in S_{2n}$ and $1\leq c\leq 2n-1$, $\mathcal{A}_{\pi}=\mathcal{A}_{\pi_c}$.
\end{claim}
\begin{proof}
By Claim~\ref{cyclic} in the previous section, for each $1\leq i\leq 2n-1$, $\tilde{M}_{\pi_c}(i)=\tilde{M}_{\pi}(i+c)$, i.e. the orders $\psi_{\pi}$ and $\psi_{\pi_c}$ are \textit{cyclically equivalent}. The claim follows.
\end{proof}
We now establish a lemma which states that for every $1\leq c\leq 2n-1$, there exists a $\sigma\in S_{2n}$ that is centered at the edge $e=\{\sigma(c),\sigma(2n)\}$.
\begin{lemma}\label{relabel}
For every $1\leq c\leq 2n-1$, there exists a $\sigma\in S_{2n}$ that is centered at the edge $\{\sigma(c),\sigma(2n)\}$.
\end{lemma}
\begin{proof}
We begin by noting that it will be sufficient to prove the claim for $c=2n-1$, as Claim \ref{obs} will then imply that it holds for every value of $c$. 

Consider $\pi\in S_{2n}$. Suppose first that $r\leq n-2$ and let $\pi$ be centered at some edge $e_0$. Next, suppose $I=(e_r,\ldots,e_0)$ is the $r$-interval in $\psi_{\pi}$ ending in the edge $e_0$. Clearly, $\{e_{r-1},\ldots,e_0\}\in \mathcal{A}_{\mu}$ but $\{e_r,\ldots,e_1\}\notin \mathcal{A}_{\mu}$. As $r\leq n-2$, $I$, which is an interval of length $r+1$, is a matching by Claim \ref{goodness}. Now any permutation $\sigma$ that satisfies $e_{\sigma}^j(2n-1)=e_j$ for each $0\leq j\leq r$ will be centered at $\{\sigma(2n-1),\sigma(2n)\}$. As an example, let $e_k=\{x_k,y_k\}$ for each $0\leq k\leq r$ and define the permutation $\sigma\in S_{2n}$ as follows. For each $1\leq k\leq r$, let $\sigma(k)=x_k$ and $\sigma(2n-1-k)=y_k$. Also, let $\sigma(2n-1)=x_0$ and $\sigma(2n)=y_0$. Finally, for each $r+1\leq j\leq 2n-r-2$, let $\sigma(j)=j$. It is clear that $\sigma$ is centered at $\{\sigma(2n-1),\sigma(2n)\}$.

Next, let $r=n-1$. We may assume without loss of generality that $\pi$ is centered at some edge in $\tilde{M}_{\pi}(2n-1)$. (If $e\in \tilde{M}_{\pi}(c)$ for some $c\in [2n-2]$, consider the permutation $\pi_c$ instead, using Claim \ref{obs}.) Suppose then that $\pi$ is centered at $e=e^{n-j}_{\pi}(2n-1)=\{\pi(n-j),\pi(n+j-1)\}$ for some $j\in [n-1]$. Let $I$ be the interval of length $n$ beginning at edge $e$. More formally,

$$I=(\underbrace{e^{n-j}_{\pi}(2n-1),\ldots,e^0_{\pi}(2n-1)}_{n-j+1 \text{ edges from } \tilde{M}_{\pi}(2n-1)},\underbrace{e^{n-1}_{\pi}(1),\ldots,e^{n-j+1}_{\pi}(1)}_{j-1 \text{ edges from }\tilde{M}_{\pi}(1)}).$$

Note that $I$ is not a matching in this case, as the edges $e^{n-j}_{\pi}(2n-1)=\{\pi(n-j),\pi(n+j-1)\}$ and $e^{n-j+1}_{\pi}(1)=\{\pi(n-j+2),\pi(n+j-1)\}$, both of which lie in $I$, share the vertex $\pi(n+j-1)$. Note also that the vertex $\pi(n-j+1)$ does not lie in any edge of $I$.

Let $F=(f_0,\ldots,f_{n-1})$, where $f_k=\{x_k,y_k\}$ is the edge in position $k+1$ of the interval $I$ for each $0\leq k\leq n-1$. We define $\sigma$ as follows. For each $1\leq k\leq n-2$, let $\sigma(k)=x_k$ and let $\sigma(2n-1-k)=y_k$. Similarly, let $\sigma(2n-1)=x_0$ and $\sigma(2n)=y_0$. This gives $e^k_{\sigma}(2n-1)=f_k$ for each $0\leq k\leq n-2$. Finally, let $\sigma(n-1)=\pi(n-j+1)$ and $\sigma(n)=\pi(n-j+2)$. This implies $e^{n-1}_{\sigma}(2n-1)=\{\pi(n-j+1),\pi(n-j+2)\}$. Now the interval of length $n-1$ that ends in $e^0_{\sigma}(2n-1)$ is a matching that lies in $\mathcal{A}$; on the other hand, the interval of length $n-1$ that ends in $e^1_{\sigma}(2n-1)$ does not lie in $\mathcal{A}$, as it has empty intersection with the matching in $\mathcal{A}_{\pi}$ that ends in the edge $e=e^{n-j}_{\pi}(2n-1)$. As $\sigma$ is saturated, this implies that it must be centered at $\{\sigma(2n-1),\sigma(2n)\}$.
\end{proof}
Now, for $\sigma\in S_{2n}$ and $1\leq j\leq 2n-1$, let $T_j(\sigma)$ be the permutation obtained from $\sigma$ by swapping the elements in positions $j$ and $j+1$ of $\sigma$. Similarly, for $\sigma\in S_{2n}$ and $1\leq j\leq n-1$, let $R_j(\sigma)$ be the permutation obtained from $\sigma$ by swapping the elements in position $j$ and $2n-1-j$ (see Figure~\ref{fig:transpositions}). Note that when $j=n-1$, $R_{n-1}(\sigma)=T_{n-1}(\sigma)$.

\begin{figure}[ht]
\begin{center}
\includegraphics{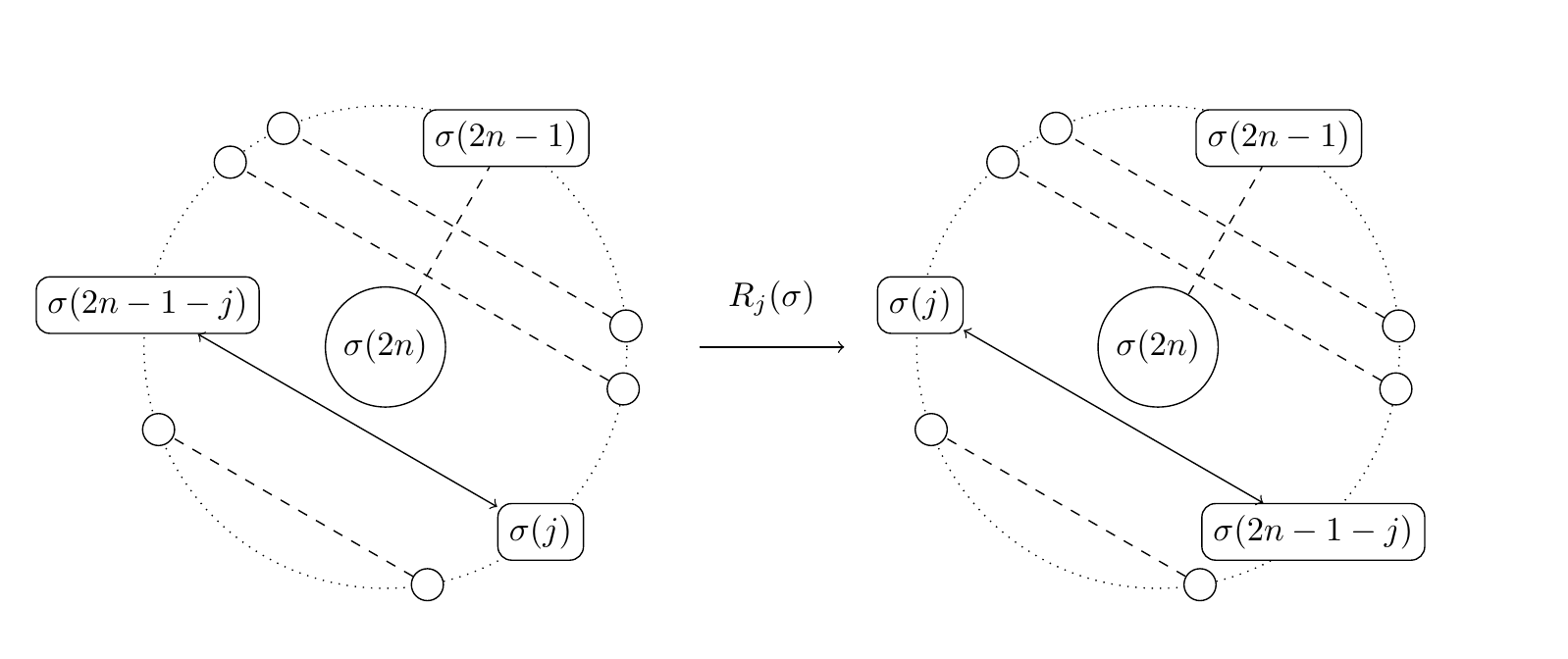}
\end{center}
\caption{An illustration of the $R_j(\sigma)$ operation.}
\label{fig:transpositions}
\end{figure}

We prove the following lemma about the $R_{j}$ operations.

\begin{lemma}\label{swaps}
For $\sigma\in S_{2n}$, and $1\leq j\leq n-1$, if $\sigma$ is saturated and centered at the edge $\{\sigma(2n-1),\sigma(2n)\}$, then if the permutation $R_{j}(\sigma)$ is saturated, it will also be centered at $\{\sigma(2n-1),\sigma(2n)\}$.
\end{lemma}
\begin{proof}
Without loss of generality, let $\sigma$ be the identity permutation, so $\sigma$ is centered at $\{2n-1,2n\}$. For $1\leq j\leq n-1$, let $\mu=R_{j}(\sigma)$. In this case, $\tilde{M}_{\sigma}(2n-1)=\{e^{n-1}_{\sigma}(2n-1),\ldots,e^0_{\sigma}(2n-1)\}$. In particular, for $1\leq j\leq n-1$, $e^j_{\sigma}(2n-1)=\{j,2n-1-j\}$. It is trivial to note that $e^j_{\mu}(2n-1)=e^j_{\sigma}(2n-1)$. In particular, this implies that $\{e^r_{\mu}(2n-1),\ldots,e^1_{\mu}(2n-1)\}\notin \mathcal{A}_{\mu}$ and $\{e^{r-1}_{\mu}(2n-1),\ldots,e^0_{\mu}(2n-1)\}\in \mathcal{A}_{\mu}$. As $\mu$ is saturated, this means that it must be centered at $\{2n-1,2n\}$.
\end{proof}
We now proceed to prove the following lemma about the $T_j$ operations.
\begin{lemma}\label{transpositions}
If every permutation $\mu\in S_{2n}$ is saturated, then for $\sigma\in S_{2n}$ and $1\leq j\leq 2n-1$, if $\sigma$ is centered at edge $e=\{\sigma(c),\sigma(2n)\}\in E$ for some $c\in [2n-1]$, then the permutation $T_j(\sigma)$ is also centered at $e$.
\end{lemma}
We begin by observing that we may assume, without loss of generality (relabeling if necessary), that $\sigma$ is the identity permutation and is centered at the edge $\{2n-1,2n\}$. We now split the proof of the lemma into two smaller claims, depending on the value of $j$. The first claim deals with the case $j\in [n]\cup \{2n-1\}$. Note that a slightly weaker hypothesis (assuming saturation only for $T_j(\sigma)$ and not for every permutation in $S_{2n}$) suffices in this case.
\begin{claim}\label{smallj}
For $j\in [n]\cup \{2n-1\}$, if $T_j(\sigma)$ is saturated, it is centered at $e=\{2n-1,2n\}$.
\end{claim}
\begin{proof}[Proof of Claim \ref{smallj}]
Note that when $j=2n-1$, the claim follows trivially, so let $1\leq j\leq n$. We note that $\tilde{M}_{\sigma}(2n-1)=\{\{n-1,n\},\{n-2,\ldots,n+1\},\ldots,\{1,2n-2\},\{2n-1,2n\}\}$ and $\tilde{M}_{\sigma}(1)=\{\{n,n+1\},\{n-1,n+2\},\ldots,\{2,2n-1\},\{1,2n\}\}$.

If $j=n-1$, then we know that $T_j(\sigma)=R_{j}(\sigma)$, so we are done by Lemma \ref{swaps}. If $j=n$, we have $e_{\mu}^{n-1}(1)=e_{\sigma}^{n-1}(1)=\{n,n+1\}$. This implies that the $r$-intervals that begin in the edges $\{2n-1,2n\}$ and $\{n,n+1\}$ are the same in both $\sigma$ and $\mu$. In particular, as $\mu$ is saturated, this means that it is centered at $\{2n-1,2n\}$. Now for each $1\leq j\leq n-2$, define the interval $I_j$ as follows.
$$I_j=\{e^{j-1}_{\sigma}(2n-1),\ldots,e^0_{\sigma}(2n-1),e^{n-1}_{\sigma}(1),\ldots,e^{j+1}_{\sigma}(1)\}.$$

Note that $I_j$ is an interval of length $n-1$ in $\psi_{\sigma}$. Also, as $j\in e^j_{\sigma}(2n-1)\cap e^{j-1}_{\sigma}(1)$ and $j+1\in e^{j+1}_{\sigma}(2n-1)\cap e^{j}_{\sigma}(1)$, $I_j$ is also compatible with $\mu$. Now there is some $A\in \mathcal{A}_{\sigma}$ such that $A\subseteq I_j$ and hence $\mu$ is centered at some $e'\in I_j$. We show that $e'=e=\{2n-1,2n\}$. Suppose not. Suppose $e'\in \tilde{M}_{\sigma}(1)$ and let $B=\tilde{M}_{\mu}(1)\setminus \{1,2n\}$. Clearly $e'\in B$. Also, as $e^{j-1}_{\sigma}(1)=\{j,2n+1-j\}$ and $e^{j}_{\sigma}(1)=\{j+1,2n-j\}$, we have $\{j+1,2n+1-j\},\{j,2n-j\}\in B$. Now let $C=\tilde{M}_{\sigma}(2n-1)\setminus \{n-1,n\}$. We have $\{j,2n-1-j\},\{j+1,2n-j-2\}\in C$, which implies that $B\cap C=\emptyset$. As there exist $C'\subseteq C$ and $B'\subseteq B$ such that $|C'|=|B'|=r$ and $C',B'\in \mathcal{A}$, this is a contradiction. A similar argument works when $e'\in \tilde{M}_{\sigma}(2n-1)$, and this completes the proof of the claim.
\end{proof}

\begin{claim}\label{bigj}
If every permutation $\mu\in S_{2n}$ is saturated, then for $n+1\leq j\leq 2n-2$, $T_j(\sigma)$ is also centered at $e=\{2n-1,2n\}$.
\end{claim}
\begin{proof}
First, let $n+1\leq j\leq 2n-3$ and let $j'=2n-2-j$. Then $1\leq j'\leq n-2$. It is not hard to observe that $\mu=R_{j'}(R_{j'+1}(T_{j'}(R_{j'}(R_{j'+1}(\sigma)))))$ (see Figure~\ref{fig:big-j}). Now, using Lemma \ref{swaps} and Claim \ref{smallj}, we can conclude that $\mu$ is centered at $\{2n-1,2n\}$.

\begin{figure}[ht]
\begin{center}
\includegraphics[scale=0.85]{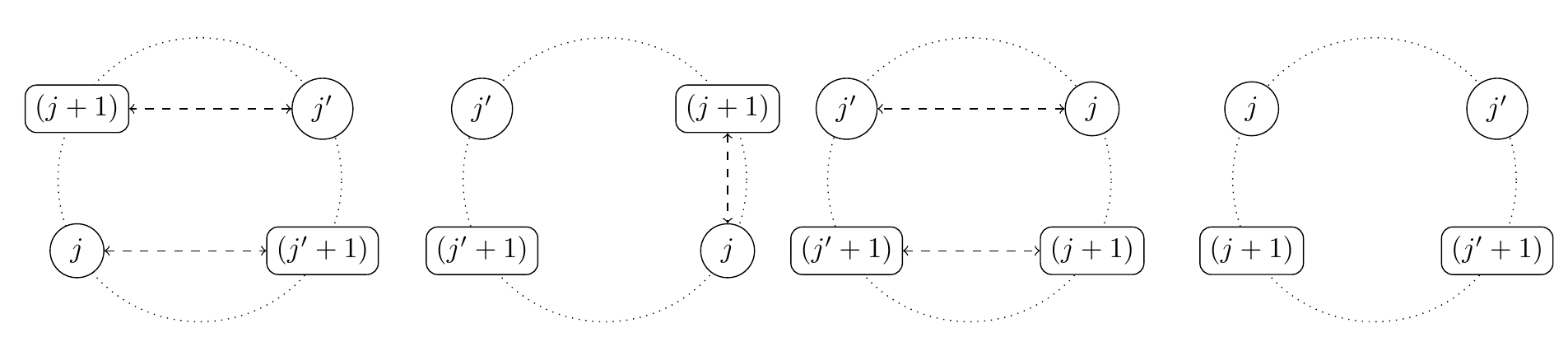}
\end{center}
\caption{A schematic showing the operations $R_{j'}(R_{j'+1}(T_{j'}(R_{j'}(R_{j'+1}(\sigma)))))$.}
\label{fig:big-j}
\end{figure}

Finally, let $j=2n-2$. Now, for $i=2n-3$, let $\mu_i=T_i(\sigma)$ and for each $1\leq i\leq 2n-4$, let $\mu_i=T_i(\mu_{i+1})$. We know, using the result for the case $1\leq j\leq 2n-3$, that $\mu_1$ is centered at $\{2n-1,2n\}$. However, it can be seen that for each $1\leq i\leq 2n-1$, $\mu(i)=\mu_1(i+1)$. Using Claim \ref{obs} with $c = 1$ and $\pi = \mu_1$, this means that $\mathcal{A}_{\mu}=\mathcal{A}_{\mu_1}$, thus implying that $\mu$ is centered at $\{2n-1,2n\}$.
\end{proof}
It is clear that Lemma \ref{transpositions} follows from Claims \ref{smallj} and \ref{bigj}. By Lemma~\ref{relabel}, we know that there is a $\sigma$ that is centered at $e=\{\sigma(c),\sigma(2n)\}$ for some $c\in [2n-1]$. Now using Lemma~\ref{transpositions}, we have that every permutation $S_{2n}$ is centered at $e$. As every matching of size $r$ (in particular, every matching containing the edge $e$) is compatible with some permutation in $S_{2n}$, it follows that $\mathcal{A}=\mathcal{A}_e$, as required.
\\
\\
\textbf{Remark:} A recent graph-theoretic generalization of the Erd\H{o}s--Ko--Rado theorem (see \cite{bh}, \cite{hs}, \cite{hst}, \cite{ht} and \cite{hk} for some examples) defines the Erd\H{o}s--Ko--Rado property for families of independent sets of a graph in the following manner. Let $\mathcal{I}^r(G)$ be the family of all independent sets of size $r$ of a graph $G$, and for a vertex $v\in G$, let $\mathcal{I}^r_v(G)$ be all sets in $\mathcal{I}^r(G)$ containing $v$, called a star in $\mathcal{I}^r(G)$ centered at vertex $v$. We say that $\mathcal{I}^r(G)$ is \textit{EKR} if for any $\mathcal{A}\subseteq \mathcal{I}^r(G)$, if $\mathcal{A}$ is intersecting, then $|\mathcal{A}|\leq \textrm{ max }_{v\in V(G)}|\mathcal{I}^r_v(G)|$. Also, we say that $\mathcal{I}^r(G)$ is \textit{strictly EKR} if it is EKR and every intersecting family of maximum size is a star. Let $K(n,k)$ be the Kneser graph as defined in Section \ref{baranyai}, and let $\overline{K(n,k)}$ be its complement, i.e. $V(\overline{K(n,k)})={[n] \choose k}$ and for any $A,B\in {[n] \choose k}$, there is an edge between $A$ and $B$ if and only if $A\cap B\neq \emptyset$. The following corollary is an immediate consequence of Theorem \ref{mainthm}.
\begin{cor}\label{knesercomp}
$\mathcal{I}^r(\overline{K(2n,2)})$ is strictly EKR for every $r\leq n-1$.
\end{cor}